\documentclass{proc-l}
\usepackage{amssymb,amsmath,amsfonts,mathtools}                          
\usepackage{graphicx}                                          
\usepackage[cmtip,all]{xy}                            
\newcommand*{\id}{{\mathrm{id}}}

\newcommand*{\GL}{{\mathrm{GL}}}     
\newcommand*{\U}{{\mathrm{U}}}     
                          
\newcommand*{\R}{{\mathbb R}}                               
\newcommand*{\CC}{{\mathbb{C}}}

\newcommand*{\Gr}{{\mathrm{Gr}}}                            
                            
\newcommand*{\Hb}{{\mathbb H}}                              
\newcommand*{\Pb}{{\mathbb P}}                              
\newcommand*{\Qb}{{\mathbb Q}}

\newcommand*{\Sb}{{\mathbb S}}                              
                              
\newcommand*{\Zb}{{\mathbb Z}}

\newcommand*{\pa}{{\partial}}

\newtheorem{theorem}{Theorem}[section]
\newtheorem{lemma}[theorem]{Lemma}    
\newtheorem{corollary}[theorem]{Corollary} 

\theoremstyle{definition}
\newtheorem{definition}[theorem]{Definition}

\theoremstyle{remark}
\newtheorem{remark}[theorem]{Remark}

\numberwithin{equation}{section}

\begin{document}

\title[Computing the Maslov index]{Computing the Maslov index for large systems} 
\author[Beck]{Margaret Beck}                                                       
\address{Maxwell Institute for Mathematical Sciences                 
and School of Mathematical and Computer Sciences,            
Heriot-Watt University, Edinburgh EH14 4AS, UK and Boston University, Boston
MA~02215, USA}
\email{mabeck@math.bu.edu}   
\thanks{}  
\author[Malham]{Simon J.A. Malham}                  
\address{Maxwell Institute for Mathematical Sciences                 
and School of Mathematical and Computer Sciences,            
Heriot-Watt University, Edinburgh EH14~4AS, UK}                 
\email{S.J.Malham@ma.hw.ac.uk}                   
\thanks{}                  
\date{6th November 2013}
\commby{}

\begin{abstract}          
We address the problem of computing the Maslov index for
large linear symplectic systems on the real line. The Maslov index measures
the signed intersections (with a given reference plane) of a path of
Lagrangian planes. The natural chart parameterization
for the Grassmannian of Lagrangian planes is the space of real 
symmetric matrices. Linear system evolution induces a Riccati evolution in the
chart. For large order systems this is a practical approach
as the computational complexity is quadratic in the order. 
The Riccati solutions, however, also exhibit singularites
(which are traversed by changing charts). 
Our new results involve characterizing these Riccati singularities
and two trace formulae for the Maslov index as follows. 
First, we show that the number of singular eigenvalues of the 
symmetric chart representation equals the dimension of intersection 
with the reference plane. 
Second, the Cayley map is a diffeomorphism from the 
space of real symmetric matrices to the manifold of unitary
symmetric matrices. We show the logarithm of the 
Cayley map equals the arctan map (modulo $2\mathrm{i}$) 
and its trace measures the 
angle of the Langrangian plane to the reference plane. 
Third, the Riccati flow under the Cayley map induces a flow   
in the manifold of unitary symmetric matrices. Using the 
natural unitary action on this manifold, we pullback
the flow to the unitary Lie algebra and monitor its trace. 
This avoids singularities, and is a natural robust procedure. 
We demonstrate the effectiveness of these approaches
by applying them to a large eigenvalue problem.
We also discuss the extension of the Maslov index 
to the infinite dimensional case.
\end{abstract}            

\maketitle

\section{Introduction}
Our goal is to compute the Maslov index for non-autonomous, symplectic,
linear differential operators of the form
\begin{equation*}
\pa-A
\end{equation*}
with domain the real line $\R$. Here $A\colon\R\to\mathfrak{sp}(\R^{2n})$,
the symplectic Lie algebra of real $2n\times 2n$ matrices. We are particularly
interested in the case when $n$ is large. We assume that the matrix
$A$ is constant in the far-field limits of the domain $\R$ 
(the two limits need not be the same). The Maslov index is an integer valued 
topological invariant that measures the following occurances. 
Consider the set of (necessarily) $n$ solutions that decay exponentially fast, say, 
to the left far-field. These solutions constitute a path of frames, indeed a path
in the Stiefel manifold of symplectic frames $\mathrm{V}(\R^{2n})$. Each frame
spans a plane in the manifold of Lagrangian planes $\Lambda(\R^{2n})$ known as the 
Lagrangian Grassmannian. The fundamental group of $\Lambda(\R^{2n})$ 
is $\pi_1\bigl(\Lambda(\R^{2n})\bigr)=\mathbb Z$. 
We evolve the Lagrangian plane from the left far-field towards the right far-field and count the 
number of intersections this Lagrangian plane path has with a given reference plane. 
The dimension of each intersection generates the 
multiplicity contribution to the Maslov index, 
and the direction of passage generates its sign contribution. 
There are various equivalent ways to formulate the definition of the 
Maslov index more precisely, and we refer to \cite{Furutani} for more details. 
This picture can be succinctly summarized in terms 
of cohomological classes of paths; see Arnol'd~\cite{Arnold}.

One of the main motivations for studying operators of the form $\pa-A$
is that they arise in the context of eigenvalue problems associated 
with the stability of solutions of travelling waves. 
For example, suppose one is interested in determining 
the qualitative behaviour of solutions of the reaction-diffusion system
\begin{equation*}
\pa_tU=B\pa_x^2U+F(U), 
\end{equation*}
where $U\colon\mathbb{R}\times\mathbb{R}^+\to\mathbb{R}^n$ and
$B$ is a diagonal matrix of positive diffusion coefficients.
A key step is to determine the existence of any stationary solutions, 
say $U_*(x)$, and their stability properties. The reason is that
stable solutions attract nearby data and therefore provide 
a qualitative description of the types of behaviours that
one would expect to observe as the system evolves. 
In order to determine if $U_*$ is stable, one can linearize the 
reaction-diffusion system about $U_*$ to obtain 
\begin{equation*}
H\coloneqq B\partial_x^2+V(x),
\end{equation*}
where $V(x)\coloneqq \mathrm{d}F(U_*(x))$ is the potential. 
In general, a necessary condition for $U_*$ to be stable is that 
the spectrum of this linear operator, in an
appropriate function space such as $L^2(\mathbb{R})$, 
be contained in the closed left half plane. 
The most difficult part of the spectrum to analyze is 
typically the point spectrum, or eigenvalues, 
which can be determined by studying
\begin{equation*}
(H-\lambda)U=O.
\end{equation*}
We assume $V$ is selfadjoint and set 
$q\coloneqq(\mathrm{Re}\,U,\mathrm{Im}\,U)^{\text{\tiny{T}}}$.
Then if we set $p\coloneqq\pa_xq$
we can rewrite this eigenvalue problem as a first order system
(hereafter $\pa\coloneqq\pa_x$)
\begin{equation*}
(\pa-A)\begin{pmatrix}q\\p\end{pmatrix}=O,
\end{equation*}
where $A=A(x,\lambda)$ is the corresponding symplectic coefficient matrix.
Eigenvalues then correspond to values of $\lambda$ for 
which solutions $(q,p)^{\text{\tiny{T}}}$ exist that decay 
to zero as $x\to\pm\infty$. This asymptotic behavior is 
equivalent to requiring 
that the solution lie in the intersection of two particular 
Lagrangian subspaces, namely the stable and unstable manifolds 
associated with $x=+\infty$ and $x=-\infty$, respectively. 
Thus, being able to determine such intersections, via the 
Maslov index (and a matrix Riccati equation as we will see), 
would provide insight into the stability of  
$U_*$. 
Such a framework is not limited to reaction-diffusion equations. 
It also applies to a variety of other examples, including higher-order 
equations such as the dispersive KdV equation. See Section~\ref{sec:example}
and examples in Chardard \textit{et al.\/ }~\cite{CDB2}.

There is a large body of existing literature on determining the 
spectral properties of the types of linear operators mentioned above. 
Perhaps the most widely used tool is the Evans function (Alexander,
Gardner and Jones~\cite{AGJ}, 
Sandstede~\cite{S:review}, Deng and Nii~\cite{DengNii}, 
Sandstede and Oh~\cite{SO}),
which is constructed for systems governed by operators of
the form $\pa-A$.
The use of the Maslov index in such contexts is relatively recent;
see Jones~\cite{J88}, Chardard \textit{et al.\/ }~\cite{CDB2},
Deng and Jones~\cite{DengJones}, Jones \textit{et al.\/ }~\cite{JLM}.
A limitation of the majority of these results 
is that they apply only to problems in one spatial dimension. 
As a result, there has been recent interest in developing techniques 
that apply to arbitrary spatial dimensions. Notable steps in this 
direction include Deng and Nii~\cite{DengNii} and 
Sandstede and Oh~\cite{SO}, involving the Evans function, 
and Deng and Jones~\cite{DengJones}, involving the Maslov index. However, 
there has not yet been an application of these abstract results to 
concrete examples, where the stability of a specific stationary solution 
is actually computed. One potential consequence of 
this current work is to provide a mechanism to do just 
that. Much of the theoretical framework herein 
in principle applies to arbitrary 
symplectic Hilbert spaces, such as those of the 
multidimensional case considered in Deng and Jones~\cite{DengJones}. 
Indications of this are given in Section~\ref{sec:conclu}.

Intuitively, the Maslov index can provide information about stability 
in the following way. Consider the reaction-diffusion system above
with $U\in\mathbb{R}$. This is just a Sturm-Liouville problem, 
and it can be recast in terms of Pr\"ufer coordinates,
which are polar coordinates for $(q,p)^{\text{\tiny{T}}}$. 
One can then determine the dynamics of the phase and prove 
a so-called oscillation theorem. This says that, 
as the eigenvalue parameter is increased, the amount
of oscillation in the phase of the corresponding solution decreases. 
One thus obtains a result concerning the sequence of eigenvalues 
$\lambda_0 > \lambda_1 > \dots$ and corresponding eigenfunctions 
$U_0, U_1, \dots$, which says that $U_k$ has exactly $k$ zeros. 
This is particularly useful when 
$\pa-A$ arises from a system such as the reaction-diffuson system above 
by linearizing about the stationary solution $U_*$. In this case, 
$\pa_xU_*$ is an eigenfunction associated with the eigenvalue $\lambda=0$. 
If, for example, $U_*$ is a pulse 
then $U_*'$ has exactly one zero. Hence, there is exactly one eigenvalue with 
positive real part and the wave is spectrally unstable. Similarly,
if $U_*$ is a front 
then the linearization has no positive eigenvalues and the wave is spectrally stable. 
The Maslov index can be thought of as a higher-dimensional
analogue of this oscillation index that applies to a more 
general class of equations. The winding, or oscillation, in the 
general case comes from the fact that the fundamental group of 
the space of Lagrangian planes is $\mathbb{Z}$. 
Note that the Maslov index and the Evans function provide 
stability information in a fundamentally different way. 
The latter is an analytic function whose zeros correspond to eigenvalues, 
and thus in principle one can locate the eigenvalues by locating its zeros. The former, 
however, allows one to count, but not locate, the eigenvalues.

Computing the Maslov index associated with $\pa-A$ will in general 
involve numerical integration. Recently Chardard \textit{et al.}~\cite{CDB2}
demonstrated how to numerically compute the Malsov index for solitary waves. 
To circumvent the inherent numerical instabilities associated with computing
the path of symplectic frames where the frame vectors have distinct
exponential growth rates, Chardard \textit{et al.}
lifted the path in the Stiefel manifold of symplectic frames $\mathrm{V}(\R^{2n})$
to a path in the symplectic exterior algebra $\bigwedge^n\mathrm{V}(\R^{2n})$.
This path has a single exponential growth rate that can be scaled out
and can be robustly numerically computed. The number of parameters
required to represent $\bigwedge^n\mathrm{V}(\R^{2n})$ is $2n$ choose $n$,
which grows exponentially with $n$. The Pl\"ucker relations between the 
parameters can be used to reduce the dimension of the system. 
A natural resolution is to apply the ortho-symplectic integration 
proposed by Chardard \textit{et al.} or indeed a symplectic version 
of the continuous orthogonalization method proposed 
by Humpherys and Zumbrun~\cite{HZ}. 

Here we pursue another natural approach which
utilizes the principle fibre bundle structure
$\GL(\R^n)\to \mathrm{V}(\R^{2n})\to \Lambda(\R^{2n})$.
Indeed we project $\mathrm{V}(\R^{2n})$ onto the 
Lagrangian Grassmannian $\Lambda(\R^{2n})$ by equivalencing by 
all general linear real rank $n$ transformations $\GL(\R^n)$.
Following Ledoux \textit{et al.\/ }~\cite{LMT}, the flow 
governed by $\pa-A$ on $\mathrm{V}(\R^{2n})$ induces a Riccati
flow in a given chart and top cell $\Lambda^0(\R^{2n})$ of the 
Lagrangian Grassmannian $\Lambda(\R^{2n})$. The Riccati
flow exhibits singularities but these just indicate the 
chart chosen $\Lambda^0(\R^{2n})$ is no longer the best 
representative chart. However, crucially, when integrating
Riccati equations in the context of eigenvalue problems as
described above, indeed for non-selfadjoint problems,
Veerle Ledoux recognized that the location (in the complex 
spectral parameter plane) of singularities 
in the components of the Riccati solution
converged to eigenvalues in the far-field limit. See 
Ledoux and Malham~\cite{LM} and Aljasser~\cite{Aljasser} for
more details. Our results here clarify and confirm these observations
in the symplectic context. 

Note that the 
Lagrangian Grassmannian chart $\Lambda^0(\R^{2n})$ is 
isomorphic to vector space of real symmetric matrices,
and the Cayley map is a diffeomorphism from the 
space of real symmetric matrices to the manifold of unitary
symmetric nonsingular matrices (see Section~\ref{sec:prelim}). 
From Arnold~\cite{Arnold} we know the logarithm of 
the determinant of the Cayley map measures the angle 
of the Langrangian plane to the reference plane. 
What we prove that is new is:
\begin{enumerate}
\item The number of singular eigenvalues in the Riccati flow on
$\Lambda^0(\R^{2n})$ equals the dimension of intersection with the 
reference plane, which equals the unsigned integer jump in 
the Maslov index; 
\item The logarithm of the Cayley map equals the arctan map
(modulo $2\mathrm{i}$) and its trace computes the angle between
the evolving Langrangian and reference planes. Hence in principle
we could push forward the Riccati flow under the arctan map and monitor 
its trace;
\item The Riccati flow under the Cayley map induces a flow   
in the manifold of unitary symmetric nonsingular matrices. Using the 
natural unitary action on this manifold, we pullback
the flow to the unitary Lie algebra and monitor its trace. 
This avoids singularities, and is a natural robust procedure
for monitoring the Maslov index. 
\end{enumerate}

Our paper is organised as follows. In Section~\ref{sec:prelim}
we fix our notation, define the Lagrangian Grassmannian and 
introduce known results that we will need. We briefly discuss
in Section~\ref{sec:TFM} the well-known method of changing 
coordinates to standardize the reference Lagrangian plane we 
seek intersections with. 
In Section~\ref{sec:Riccati} we derive the Riccati flow
prescribing the evolution by $\pa-A$ of a Lagrangian plane 
and establish the connection between singularities in the flow 
and intersections with the reference plane. We provide the
details of how to pull back the flow to the unitary Lie algebra
in Section~\ref{sec:UnitaryRiccati}. We compute the Maslov index 
for a large (high order) eigenvalue problem in
Section~\ref{sec:example} and demonstrate the computational advantages
of the methods we have derived/propose. Finally in
Section~\ref{sec:conclu} we discuss the extension to the 
infinite dimensional case.

\section{Lagrangian Grassmannian}\label{sec:prelim}
We introduce the objects of our study and
fix our notion as follows. We denote by $\mathrm{Sp}(\R^{n})$, 
$\GL(\CC^n)$,  $\mathrm{O}(\R^n)$ and $\mathrm{U}(\CC^n)$, respectively, 
the usual symplectic, general linear, orthogonal and unitary Lie groups
of rank $n$ matrices over the fields indicated. We use 
$\mathfrak{sp}(\R^n)$ and $\mathfrak{u}(\CC^n)$ to denote the 
corresponding symplectic and unitary Lie algebras. We denote by:
\begin{itemize}
\item $\mathrm{V}(\R^{2n})$, the Stiefel manifold of symplectic $n$-frames in $\R^{2n}$;
\item $\Lambda(\R^{2n})$, the Lagrangian Grassmann manifold of $n$-planes in $\R^{2n}$;
\item $\mathrm{U}_{\mathrm{sym}}(\CC^n)$, the manifold of unitary 
symmetric nonsingular matrices in $\CC^n$.
\end{itemize}
Our subsequent analysis mainly involves these standard Lie groups and homogeneous manifolds.
Some of the important relations between them are summarized in the following diagram.
\begin{equation*}
\xymatrixcolsep{2.7pc}
\xymatrixrowsep{2.7pc}
\xymatrix{
\Lambda^0(\R^{2n})~\ar@{->}[d]_{\mathrm{Cayley}}
&\ar@{->}[l]~\mathrm{Sp}(\R^{2n})\cap\mathrm{O}(\R^{2n})~\ar@{^{(}->}[r] 
&~\mathrm{Sp}(\R^{2n})\\
\U_{\mathrm{sym}}(\CC^n)~ &\ar@{->}[l]~\U(\CC^n)\ar@{->}[u]~\ar@{^{(}->}[r] 
 &~\GL(\CC^n) \ar@{->}[u]
}
\end{equation*}
Here the vertical lift maps are isomorphisms and all the horizontal maps 
are embeddings and projections as indicated. The relations between the four groups 
on the right are given by natural isomorphic liftings of the 
respective complex groups (bottom) to their real 
counterparts (top), combined with the isomorphism 
$\U(\CC^n)\cong\GL(\CC^n)\cap\mathrm{O}(\R^{2n})$. 
The bottom left projection results from the fibering
\begin{equation*}
\mathrm{O}(\R^n)\to\U(\CC^n)\to\U_{\mathrm{sym}}(\CC^n).
\end{equation*}
The leftmost diffeomorphism, the \emph{Cayley transform}, is explained 
as follows---more details can be found in Arnol'd~\cite{Arnold}
and de Gosson~\cite{deGosson}. 
The Lagrangian Grassmannian is equivalently defined via
the fibering
\begin{equation*}
\GL(\R^n)\to\mathrm{V}(\R^{2n})\to\Lambda(\R^{2n}).
\end{equation*}
Suppose a symplectic frame in $\mathrm{V}(\R^{2n})$ is represented by
a $2n\times n$ matrix.
Charts that cover $\Lambda(\R^{2n})$ are represented by the set of $2n\times n$ matrices 
with distinct $n\times n$ submatrices given by the identity. There are
$2n$ choose $n$ charts. For example if $q$ and $p$ are square $n\times n$ matrices
forming a frame in $\mathrm{V}(\R^{2n})$ we can choose a chart on
$\Lambda(\R^{2n})$ by projecting as follows,
\begin{equation*}
\begin{pmatrix} q\\ p \end{pmatrix}\mapsto\begin{pmatrix} \id\\ s \end{pmatrix},
\end{equation*}
where $s\colon q\mapsto p$ is necessarily symmetric due to symplectic
condition on the frame. 
We observe the Lagrangian Grassmannian is thus $\frac12 n(n+1)$ dimensional. 
The set of Lagrangian planes $\Lambda^0(\R^{2n})$ represented in this chart 
are dense in $\Lambda(\R^{2n})$, and we shall denote
$\Lambda^0(\R^{2n})$ as the top \emph{Schubert cell} in $\Lambda(\R^{2n})$.
The Lagrangian Grassmannian $\Lambda(\R^{2n})$ is the disjoint
union of all the Schubert cells (more on this presently).
As we have already indicated, $\Lambda^0(\R^{2n})$ is diffeomorphic
to the linear space of real symmetric matrices $\mathrm{D}(\R^n)$ of order $n$.
The Cayley transform $\mathrm{Cay}\colon\mathrm{D}(\R^n)\to\U_{\mathrm{sym}}(\CC^n)$ 
given by
\begin{equation*}
\mathrm{Cay}\colon s\mapsto \frac{\id-\mathrm{i}\,s}{\id+\mathrm{i}\,s}
\end{equation*}
is a diffeomorphism and plays an important role in our main results. 
From Arnol'd~\cite{Arnold} there is a natural map 
$\mathrm{Det}^2\colon\Lambda(\R^{2n})\to\mathbb{S}^1$
given by 
\begin{equation*}
\mathrm{Det}^2\coloneqq \mathrm{det}\circ\mathrm{Cay}.
\end{equation*}
We deduce the fundamental group 
$\pi_1\bigl(\Lambda(\R^{2n})\bigr)\cong\mathbb Z$. 
\begin{remark}
In coordinates, for $s\in\Lambda^0(\R^{2n})$ the map $\mathrm{Det}^2$
is straightforward via the above decomposition. However it is
well defined for any plane in $\Lambda(\R^{2n})\backslash\Lambda^0(\R^{2n})$,
i.e.\/ in any lower Schubert cell, as we see in detail in Section~\ref{sec:Riccati}.
\end{remark}

\section{Standard reference plane and total frame matrix}\label{sec:TFM}
Assume we have a path in the Stiefel manifold of symplectic frames $V(\R^{2n})$
generated by the solution flow of the non-autonomous linear differential system
on the real line $\R$,
\begin{equation*}
(\pa-A)\begin{pmatrix}q\\p\end{pmatrix}=O.
\end{equation*}
Here $q$ and $p$ are real $n\times n$ matrices
and $A\colon\R\to\mathfrak{sp}(\R^{2n})$ has the block form
\begin{equation*}
A\coloneqq\begin{pmatrix}a&b\\c&d\end{pmatrix},
\end{equation*}
where $a$, $b$, $c$ and $d$ are real $n\times n$ matrices with
$b$ and $c$ symmetric and $a=-d^{\text{\tiny{T}}}$. The path of 
symplectic frames over $\R$ induces a path of Lagrangian planes
(spanning each frame). We wish to determine the number and type of
intersections of the path of Lagrangian planes with a reference
plane. The occurance and dimension of their intersection is measured 
by the drop in rank of the $2n\times2n$ matrix
\begin{equation*}
\begin{pmatrix}q&q_0\\p&p_0\end{pmatrix},
\end{equation*}
where the left $2n\times n$ block represents our evolutionary symplectic
frame in $V(\R^{2n})$ and the right $2n\times n$ block represents the 
reference plane in $V(\R^{2n})$. Since the reference plane is fixed,
it is natural to seek a change of coordinates to carry the reference
plane to the \emph{standard reference plane} with $q_0=O$ and $p_0=\id$. 
A linear rank $2n$ transformation of coordinates that achieves this is
\begin{equation*}
\begin{pmatrix}q'&O\\p'&\id\end{pmatrix}
=\begin{pmatrix}\id&q_0p_0^{-1}\\O&p_0^{-1}\end{pmatrix}
\begin{pmatrix}q&q_0\\p&p_0\end{pmatrix},
\end{equation*}
where $q'\coloneqq q-q_0p_0^{-1}p$ and $p'\coloneqq p_0^{-1}p$.
This assumes that $p_0$ is nonsingular for the reference plane; it is
not difficult to adapt our subsequent arguments if this is not the case.
Note that \emph{rank is conserved} under this linear transformation.
\begin{definition}[Total frame matrix]
This is the $2n\times2n$ matrix given by
\begin{equation*}
\begin{pmatrix}q'&O\\p'&\id\end{pmatrix}.
\end{equation*}
Its loss of rank measures the occurance and dimension of the
intesection of the plane spanned by the frame 
$(q',p')^{\text{\tiny{T}}}\in V(\R^{2n})$ with the standard reference
plane, or equivalently, the plane spanned by 
$(q,p)^{\text{\tiny{T}}}\in V(\R^{2n})$ with 
$(q_0,p_0)^{\text{\tiny{T}}}\in V(\R^{2n})$.
\end{definition}
\begin{remark}
Since $(q,p)$ satisfy the linear system generated by $\pa-A$,
then $(q',p')$ satisfy an analogous transformed linear system.
\end{remark}
\begin{remark}
Note that, as they represent a solution frame in $\mathrm{V}(\R^{2n})$ of
a linear system, $q'$ and $p'$ (or $q$ and $p$) are finite matrices 
whose components can only become singular in either far-field limit.
\end{remark}
Henceforth we will drop `primes' and assume $q'\to q$
and $p'\to p$.

\section{Riccati flow in the chart}\label{sec:Riccati}
Recall the Lagrangian Grassmannian chart we denoted 
as the top cell $\Lambda^0(\R^{2n})$ in Section~\ref{sec:prelim}.
We derive the Riccati flow in $\Lambda^0(\R^{2n})$ generated
by the linear symplectic flow in $\mathrm{V}(\R^{2n})$.
Define $s\in\mathrm{D}(\R^n)$ as the real symmetric map 
$s\colon q\mapsto p$.
Note $q$ and $p$ satisfy the equations $\pa q=(a+bs)\,q$ and
$\pa p=(c+ds)\,q$. Using $\pa(sq)=\pa s\,q+s\pa q$ we know
$\pa s\,q=\pa p-s\,\pa q$. Direct substitution reveals
$\pa s\,q=c\,q+ds\,q-s\,(a\,q+bs\,q)$.
Hence we see that $s$ satisfies the \emph{Riccati} equation
\begin{equation*}
\pa s=c+ds-s(a+bs).
\end{equation*}
Thus $s\in\mathrm{D}(\R^n)$ prescribes the evolution of the
Lagrangian planes which can be represented in the top
cell $\Lambda^0(\R^{2n})$. Recall that $\Lambda^0(\R^{2n})$
is dense in $\Lambda(\R^{2n})$ and the set of Lagrangian planes 
in $\Lambda(\R^{2n})\backslash\Lambda^0(\R^{2n})$ are all
those respresented in the lower dimensional Schubert cells.
Components in the Riccati solution $s$ can become \emph{singular}
in finite time and this is the \emph{signature} that 
the Lagrangian plane is no longer representable in $\Lambda^0(\R^{2n})$
and moves into a lower Schubert cell. In turn this induces
a drop in rank in $q$ and thus also the total frame matrix.

More precisely, let the real eigenvalues of $s$ be $\mu_i$, $i=1,\ldots,n$.
Imagine we evolve our frame 
$(q,p)^{\text{\tiny{T}}}\in V(\R^{2n})$ under 
$\pa-A$ from one far field to the other. For all $x\in\R$
the solution components $q=q(x)$ and $p=p(x)$ are well defined, bounded 
and the frame $\bigl(q(x),p(x)\bigr)^{\text{\tiny{T}}}$ has rank $n$.
Alongside this frame, we evolve $s\colon\R\to\Lambda^0(\R^{2n})$,
i.e.\/ $s=pq^{-1}$, under its natural Riccati evolution. Note
that $s=s(x)$ is well defined on $\mathrm{Rg}\,q(x)$, the range
of the operator $q(x)$. 
We say that $\mu_i=\mu_i(x)$ is a \emph{singular eigenvalue} of
$s=s(x)$ at $x=x_0$ if $\lim_{x\to x_0}|\mu_i(x)|=\infty$.
\begin{theorem}[Equivalent Maslov index contributions]\label{th:sing}
When the Lagrangian plane intersects the standard reference plane,
the following integer valued quantites are equal (we do not concern
ourselves with the direction of crossing here):
\begin{enumerate}
\item Unsigned integer jump value in the Maslov index;
\item Rank loss in the matrix $q$;
\item Rank loss in the total frame matrix;
\item Dimension of intersection of Lagrangian and standard reference planes;
\item Number of singular eigenvalues of $s$.
\end{enumerate}
\end{theorem}
\begin{proof}
The first four equivalences can be deduced from results in Arnold~\cite{Arnold} 
and de Gosson~\cite{deGosson} .
Note the rank loss count of de Gosson~\cite[Lemma~198, Prop.~263]{deGosson} 
is double ours as we use the total frame matrix rather than a
map into $\mathrm{Sp}(\R^{2n})$. Hence our goal is to establish that the rank loss 
in $q$ equals the number of singular eigenvalues of $s\in\mathrm{D}(\R^n)$.
More precisely, we wish to show that for any $x_0\in\R$, $s(x_0)$ has 
exactly $k$ singular eigenvalues if and only if $\mathrm{Rank}\,q(x_0)=n-k$.
Recall $q$ and $p$ satisfy the properties $q^{\text{\tiny{T}}}p=p^{\text{\tiny{T}}}q$ 
and $pq^{\text{\tiny{T}}}=qp^{\text{\tiny{T}}}$,
and these imply $s=s(x)$ is symmetric, at least on $\mathrm{Rg}\,q$. 
Thus all of the eigenvalues of $s$ are real. Note that 
$\mathrm{Ker}\,p \cap \mathrm{Ker}\,q=\{O\}$ for all $x$, 
because otherwise there would exist a $v\in\R^n$ such that $v\neq O$ and 
\begin{equation*}
\begin{pmatrix}q(x)\\p(x)\end{pmatrix}v=O, 
\end{equation*}
which contradicts the fact that the rank of this matrix must be $n$. 
Next recall that for any square matrix $M$, we have
$\mathbb{R}^n=\mathrm{Ker}\,M\oplus\mathrm{Rg}\,M^{\text{\tiny{T}}}
=\mathrm{Ker}\,M^{\text{\tiny{T}}}\oplus\mathrm{Rg}\,M$. 
These facts imply that $q^{-1}\colon\mathrm{Rg}\,q\to\mathrm{Rg}\,q^{\text{\tiny{T}}}$,
$p\colon\mathrm{Rg}\,q^{\text{\tiny{T}}}\to\mathrm{Rg}\,q$ and thus 
$s\colon\mathrm{Rg}\,q\to\mathrm{Rg}\,q$.

Suppose $\mathrm{dim}\,\mathrm{Rg}\,q=n-k$, which implies that $s$ has at least $n-k$ eigenvalues 
that are not singular, and therefore at most $k$ singular eigenvalues. 
Suppose there is an additional one, which implies that $s|_{\mathrm{Ker}\,q^{\text{\tiny{T}}}}$ has an 
eigenvalue that is not singular: $s(x_0)v(x_0)=\mu(x_0)v(x_0)$, 
$|\mu(x_0)|<\infty$, and $v(x_0)\in\mathrm{Ker}\,q^{\text{\tiny{T}}}$. 
But we then have
\begin{equation*}
0=q^{\text{\tiny{T}}}v=q^{\text{\tiny{T}}}\mu v=q^{\text{\tiny{T}}}sv
=p^{\text{\tiny{T}}}qq^{-1}v=p^{\text{\tiny{T}}}v. 
\end{equation*}
This is a contradiction because $p^{\text{\tiny{T}}}(x_0)v(x_0)\neq O$ 
if $\mathrm{rank}\,(q^{\text{\tiny{T}}}p^{\text{\tiny{T}}})=n$.

Now suppose $s$ has exactly $k$ singular eigenvalues, 
and thus $n-k$ nonsingular eigenvalues. This implies
$\mathrm{dim}\,\mathrm{Rg}\,q\leqslant n-k$, and so 
$\mathrm{dim}\,\mathrm{Ker}\,q=\mathrm{dim}\,\mathrm{Ker}\,q^{\text{\tiny{T}}}\geqslant k$. 
If it was strictly greater than $k$, we could again consider the restriction 
$s|_{\mathrm{Ker}\,q^{\text{\tiny{T}}}}$ and obtain a contradiction as above.
\end{proof}
\begin{remark}[The train]
The quantities listed in Theorem~\ref{th:sing}
are equivalent to the dimension of intersection with the \emph{train};
see Arnold~\cite{Arnold}, Jones~\cite{J88} or Bose and Jones~\cite{BoseJones}.
They are also equivalent to the drop in dimension from the top to lower 
Schubert cell. Here, the dimension of a Schubert cell is the number of of real
variables required to parameterize it.
\end{remark}
\begin{remark}
We assume the points $x_0\in\R$ where the eigenvalues of $s$ are singular
are isolated. In our applications to spectral problems we will always 
restrict ourselves to regions of the spectral parameter plane where we
know this is true. Moreover, such a degeneracy could be removed via homotopy, 
and as the Maslov index is invariant under homotopy, this need not affect the count.
\end{remark}
Recall the Cayley diffeomorphism
$\mathrm{Cay}\colon\mathrm{D}(\R^n)\to\U_{\mathrm{sym}}(\CC^n)$ 
from Section~\ref{sec:prelim} and that the determinant map
composed with this is $\det\,\mathrm{Cay}\colon\mathrm{D}(\R^n)\to\Sb^1$.
Note that $\Sb^1\cong\U(\CC^1)$. In coordinates we can thus write
\begin{equation*}
\mathrm{e}^{\mathrm{i}\theta}=\mathrm{det}\,\mathrm{Cay}\,s
\end{equation*}
where $\theta\in[0,2\pi)$ parameterizes $\mathfrak{u}(\CC^1)$ and
measures the angle between the evolving Lagrangian and 
standard reference planes. First, since
$\mathrm{det}\,\exp=\exp\,\mathrm{tr}$ we find
\begin{equation*}
\mathrm{det}\,\mathrm{Cay}=\exp\,\mathrm{tr}\,\log\,\mathrm{Cay}.
\end{equation*}
Second, the function $\mathrm{Cay}(s)$
has an expansion in the algebra 
of operator power series with $\CC$-valued coefficients.
Hence by the uniqueness of power series we find
\begin{equation*}
\log\mathrm{Cay}\,s=
\log\,(\id-\mathrm{i}\,s)-\log\,(\id+\mathrm{i}\,s)
=2\mathrm{i}\,\mathrm{arctan}\,s.
\end{equation*}
Putting these two results together we have just proved the following.
\begin{lemma}[Trace formula]\label{lemma:trace}
The following maps $\mathrm{D}(\R^n)\to\U(\CC^1)$ are equivalent:
\begin{equation*}
\log\,\det\,\mathrm{Cay}=2\mathrm{i}\,\mathrm{tr}\,\mathrm{arctan}.
\end{equation*}
The angle between the evolving Lagrangian plane and the standard
reference plane is thus equivalently given in coordinates by
$\theta=-2\,\mathrm{tr}\,\mathrm{arctan}\,s$.
\end{lemma}
\begin{remark}
The trace formula implies
$\mathrm{tr}\,\mathrm{arctan}\,s=\sum_{i=1}^n\mathrm{arctan}\,\mu_i$.
We can also derive a nonlinear equation prescribing the flow of
$\mathrm{arctan}\,s$ directly.
\end{remark}
\begin{remark}[Non-selfadjoint operators]
More generally we could consider non-selfadjoint 
operators $\pa-A$ and thus coefficient matrices
$A\in\mathfrak{gl}(\CC^{n})$, the general linear Lie algebra over $\CC$.
In this case the natural manifolds of interest are 
$\mathrm{V}_{k}(\CC^n)$, the Stiefel manifold of complex $k$-frames in $\CC^n$,
and $\Gr_k(\CC^n)$, the Grassmann manifold of $k$-planes in $\CC^n$.
There is a natural fibering
$\GL(\CC^k)\to\mathrm{V}_k(\CC^{n})\to\Gr_k(\CC^n)$.
The Grassmannian $\Gr_k(\CC^n)$ has a Schubert cell structure.
The top cell is parameterized by complex $(n-k)\times k$ matrices.
We can proceed as for the symplectic case. 
Suppose $(q,p)^{\text{\tiny{T}}}$ 
represents a frame in the Stiefel manifold $\mathrm{V}_{k}(\CC^n)$,
where $q\in\CC^{k\times k}$ and $p\in\CC^{(n-k)\times k}$. 
Then $s\colon q\mapsto p$ parameterizes the top cell and satisfies
the Riccati equation $\pa s=c+ds-s(a+bs)$,
where $a\in\CC^{k\times k}$, $b\in\CC^{k\times(n-k)}$, $c\in\CC^{(n-k)\times k}$
and $d\in\CC^{(n-k)\times(n-k)}$
are the blocks of $A$ juxtaposed as for the symplectic case.
Ledoux \textit{et al.\/ }~\cite{LMT} advocated integrating
the Riccati equation for $s$ to evaluate the Evans function
\begin{equation*}
\det\begin{pmatrix}q&q_0\\p&p_0\end{pmatrix},
\end{equation*}
associated with non-selfadjoint eigenvalue problems, where
$q_0$ and $p_0$ represent the far field data. 
We can perform the same rank preserving linear transformation as we did
in Section~\ref{sec:TFM} to transform the Evans matrix 
into a form analogous to the total frame matrix; again
we reassign $q$ and $p$ as the new transformed solution variables.
Then the rank loss in $q$, rank loss in the total frame
Evans matrix and dimension of the intersection of unstable 
and reference planes, are all equivalent. We conjecture they
are also equivalent to the number of singular eigenvalues of 
the matrix $s$.
\end{remark}
\begin{remark}[Superpotentials] If 
$(q,p)^{\text{\tiny{T}}}\in\mathrm{V}_{k}(\CC^n)$,
but with $q\in\CC^{(n-k)\times k}$ and $p\in\CC^{k\times k}$, then
the map $r\colon p\mapsto q$ satisfies the Riccati equation 
$\pa r=b+ar-r(d+cr)$.
The solutions $s$ and $r$ to their respective Riccati equations 
represent superpotentials, see Bougie \textit{et al.\/ }~\cite{BGMR}.
Formally they diagonalize $\pa-A$ in the sense that 
\begin{equation*}
\begin{pmatrix}\id & r \\ s & \id\end{pmatrix}^{-1}(\pa-A)
\begin{pmatrix}\id & r \\ s & \id\end{pmatrix}
=\begin{pmatrix} \pa-(a+bs) & O \\ O&\pa-(d+cr)
\end{pmatrix}.
\end{equation*}
Having diagonalized $\pa-A$ into the block operators 
$\pa-(a+bs)$ and $\pa-(d+cr)$, we can iterate this procedure, 
separately diagonalizing $\pa-(a+bu)$ and $\pa-(d+cv)$,
and so forth. In principle, provided we can solve the Riccati equations
at each level, we can complete a full scalar-block diagonalization of $\pa-A$.
We can also Schur-triangularize. If we replace $r$ by $s^\dag$, the transformation
operator above becomes unitary and Schur-triangularizes $\pa-A$, which
is more akin to the usual factorization procedure in superpotential theory.
We can iterate this procedure in an analogous manner to that described
for diagonalization.
\end{remark}

\section{Riccati flow on the symmetric unitary manifold}\label{sec:UnitaryRiccati}
How does the Riccati flow on $\Lambda^0(\R^{2n})$ transform under
the Cayley diffeomorphism?
The answer is a related Riccati flow as follows. Setting
\begin{equation*}
u\coloneqq\frac{\id-\mathrm{i}\,s}{\id+\mathrm{i}\,s}
\end{equation*}
and computing the derivative $\pa u$ in terms of 
$\pa s$ generates the following.
\begin{lemma}\label{lem:UnitaryRiccati}
If $s\colon\R\to\mathrm{D}(\R^n)$ solves the 
Riccati equation $\pa s=c+ds-s(a+bs)$, then the 
corresponding flow $u\colon\R\to\U_{\mathrm{sym}}(\CC^n)$ is given by
\begin{equation*}
\pa u=C+Du-u(D^\ast+C^\ast u),
\end{equation*}
where $C^\ast$ and $D^\ast$ denote the complex conjugates of 
the $n\times n$ matrices 
\begin{equation*}
C\coloneqq\tfrac12\bigl(a-d-\mathrm{i}(b+c)\bigr)
\qquad\text{and}\qquad
D\coloneqq\tfrac12\bigl(a+d+\mathrm{i}(b-c)\bigr).
\end{equation*}
\end{lemma}
There is a natural transitive Lie group action 
$L\colon\U(\CC^n)\times\U_{\mathrm{sym}}(\CC^n)\to\U_{\mathrm{sym}}(\CC^n)$,
which for all unitary symmetric matrices $u$ and unitary matrices $g$ is
given by $L\colon (g,u)\mapsto gug^{\text{\tiny{T}}}$.
Note this is not an isospectral action. 
The corresponding Lie algebra action
$\ell\colon\mathfrak{u}(\CC^n)\times\U_{\mathrm{sym}}(\CC^n)\to\U_{\mathrm{sym}}(\CC^n)$
is given by
\begin{equation*}
\ell\colon(\xi,u)\mapsto\xi u-u\xi^\ast,
\end{equation*}
for any Lie algebra element $\xi\in\mathfrak{u}(\CC^n)$. Recall
for $\xi\in\mathfrak{u}(\CC^n)$ 
the complex conjugate transpose $\xi^\dag=-\xi$.
We can now establish the following.
\begin{theorem}[Pullback to the unitary Lie algebra]\label{th:unitaryflow}
The Riccati flow on $\U_{\mathrm{sym}}(\CC^n)$ has
the form $\pa u=\xi u-u\xi^\ast$
where $\xi\colon\R\times\U_{\mathrm{sym}}(\CC^n)\to\mathfrak{u}(\CC^n)$ 
is given by
\begin{equation*}
\xi\coloneqq D-\tfrac12(uC^\ast-Cu^\dag).
\end{equation*}
We can pullback the flow $u\colon\R\to\U_{\mathrm{sym}}(\CC^n)$ to
the flow $g\colon\R\to\U(\CC^n)$ on the unitary group given by $\pa g=\xi\,g$.
Hence we can pullback the flow $g\colon\R\to\U(\CC^n)$ to the 
flow $\sigma\colon\R\to\mathfrak{u}(\CC^n)$ 
on the unitary Lie algebra given by
\begin{equation*}
\pa\sigma=\mathrm{dexp}_\sigma^{-1}\circ\xi.
\end{equation*}
Here $\mathrm{dexp}_\sigma^{-1}\colon\mathfrak{u}(\CC^n)\to\mathfrak{u}(\CC^n)$
is the operator 
$\mathrm{dexp}_\sigma^{-1}\coloneqq\mathrm{ad}_\sigma/(\exp\mathrm{ad}_\sigma-\id)$,
where $\mathrm{ad}\colon\mathfrak{u}(\CC^n)\times\mathfrak{u}(\CC^n)\to\mathfrak{u}(\CC^n)$
is the Lie algebra commutator operator given in coordinates by
$\mathrm{ad}_\sigma\circ\zeta\coloneqq[\sigma,\zeta]$ for any $\sigma,\zeta\in\mathfrak{u}(\CC^n)$.
\end{theorem}
\begin{proof}
That the unitary symmetric flow on $\U_{\mathrm{sym}}(\CC^n)$ has
the form $\pa u=\xi u-u\xi^\ast$ can be checked by inspection.
This establishes the vector field governing the flow is a unitary 
Lie algebra action. The pullback of the flow via the Lie group action 
to the unitary Lie group $\U(\CC^n)$ to the flow $\pa g=\xi\,g$ is then
standard. The pullback of the unitary flow via the exponential map 
to the unitary Lie algebra $\mathfrak u(\CC^n)$ to the flow stated
is also standard; see Munthe--Kaas~\cite{HMK} or Malham and Wiese~\cite{MW}.
\end{proof}
\begin{corollary}[Trace of the Lie algebra flow]\label{cor:trace}
The angle $\theta\in\mathfrak{u}(\CC^1)$ 
between the evolving Lagrangian plane and the standard
reference plane is given by
\begin{equation*}
\theta=-2\mathrm{i}\,\mathrm{tr}\sigma+\theta_0,
\end{equation*}
where $\theta_0$ is its initial value and 
$\mathrm{tr}\colon\mathfrak{u}(\CC^n)\to\mathfrak{u}(\CC^1)$ 
is the trace map.
\end{corollary}
\begin{proof}
Using the unitary Lie group action on $\U_{\mathrm{sym}}(\CC^n)$ 
we observe that
\begin{equation*}
\mathrm{e}^{\mathrm{i}\theta}=\det\, u
=\det\,gu_0g^{\text{\tiny{T}}}
=(\det\,g)^2\det\,u_0
=\exp(2\,\mathrm{tr}\,\sigma)\det\,u_0.
\end{equation*}
Taking the logarithm and setting 
$\theta_0=-\mathrm{i}\,\log\det\,u_0$ gives the result.
\end{proof}
%
%
\begin{remark}
To compute the evolution of the angle between the evolving Lagrangian
plane and the standard reference plane we must solve the differential
equation for $\sigma$ in Theorem~\ref{th:unitaryflow} and monitor its
trace. Typically this procedure will be numerical. Fortunately, however,
numerical procedures for solving for $\sigma$ (or $g$ or
$u$) are robust in the sense that truncations of the vector field
for $\sigma$, required for implementing Runge--Kutta methods for example,
still lie in the unitary Lie algebra. Indeed this is the idea underlying
Lie group numerical methods such as Runge--Kutta Munthe--Kaas methods 
developed by Munthe--Kaas~\cite{HMK}. 
\end{remark}

\section{Example}\label{sec:example}
We apply our methods to the KdV equation from Chardard
\textit{et al.\/ }~\cite{CDB2} given by
\begin{equation*}
\pa_tU-c\pa_xU+U\pa_xU+\pa_x^3U-\pa_x^5U+\sigma\pa_x^7U=0,
\end{equation*}
where $c=71000/2159^2$ and $\sigma=0.2159$ are constants. 
This has a solitary-wave solution given by 
$U_\ast(x)=a(\mathrm{sech}^6kx+\mathrm{sech}^4kx)$ where
$a=1039500/2159^2$ and $k=(25/2159)^{1/2}$. The above equation 
is Hamiltonian, and the stability of the solitary wave can 
be determined by computing the spectrum of the Hessian of the Hamiltonian, 
i.e.\/ we consider the self-adjoint spectral problem 
$(c-U_*-\pa_x^2+\pa_x^4-\sigma\pa_x^6)U=\lambda U$. 
Setting 
$(q,p)^{\text{\tiny{T}}}\coloneqq(U,\pa_x^2U-\sigma\pa_x^4U,\pa_x^2U,
\pa_xU-\pa_x^3U+\sigma\pa_x^5U,-\pa_xU,\sigma\pa_x^3U)^{\text{\tiny{T}}}$
we find that our goal is to compute the values of the spectral
parameter $\lambda\in\R$ for which the kernel of the operator $\pa-A$
is non-empty---we call these spectral eigenvalues---where
\begin{equation*}
A(x,\lambda)=\begin{pmatrix}
0&0&0&0&-1&0\\
0&0&0&-1&-1&0\\
0&0&0&0&0&1/\sigma\\
-\lambda+c-U_\ast(x)&0&0&0&0&0\\
0&0&-1&0&0&0\\
0&-1&1&0&0&0
\end{pmatrix}.
\end{equation*}
See Chardard \textit{et al.\/ }~\cite{CDB2} for details.
For fixed values of $\lambda\in\R$ we evolve the unstable
manifold, here the Lagrangian plane parameterised by 
$s=s(x,\lambda)\in\mathrm{D}(\R^3)$, from the $x\to-\infty$ 
far field forward to $x\to+\infty$. We used the method of
Schiff and Shnider~\cite{SchiffShnider} to integrate the 
Riccati equation through singularites; also see 
Ledoux \textit{et al.\/ }~\cite{LMT}. We record the eigenvalues
$\mu_i=\mu_i(x,\lambda)$, $i=1,2,3$. Figure~\ref{fig:singeval1top}
(left panel) shows the eigenvalue of $s=s(x,\lambda)$ that becomes singular---the
other two eigenvalues of $s$ remain finite for all $(x,\lambda)\in\R^2$.
We see three curves in the $(x,\lambda)$ plane where the eigenvalue
of $s$ becomes singular as we expect. Further for large $x$ the singular
eigenvalue curves asymptotically approach fixed values---indicating 
an intersection of the evolved Lagrangian plane with the far field 
Lagrangian plane and signifying a spectral eigenvalue. 
\begin{figure}
  \begin{center}
  \includegraphics[width=5cm,height=5cm]{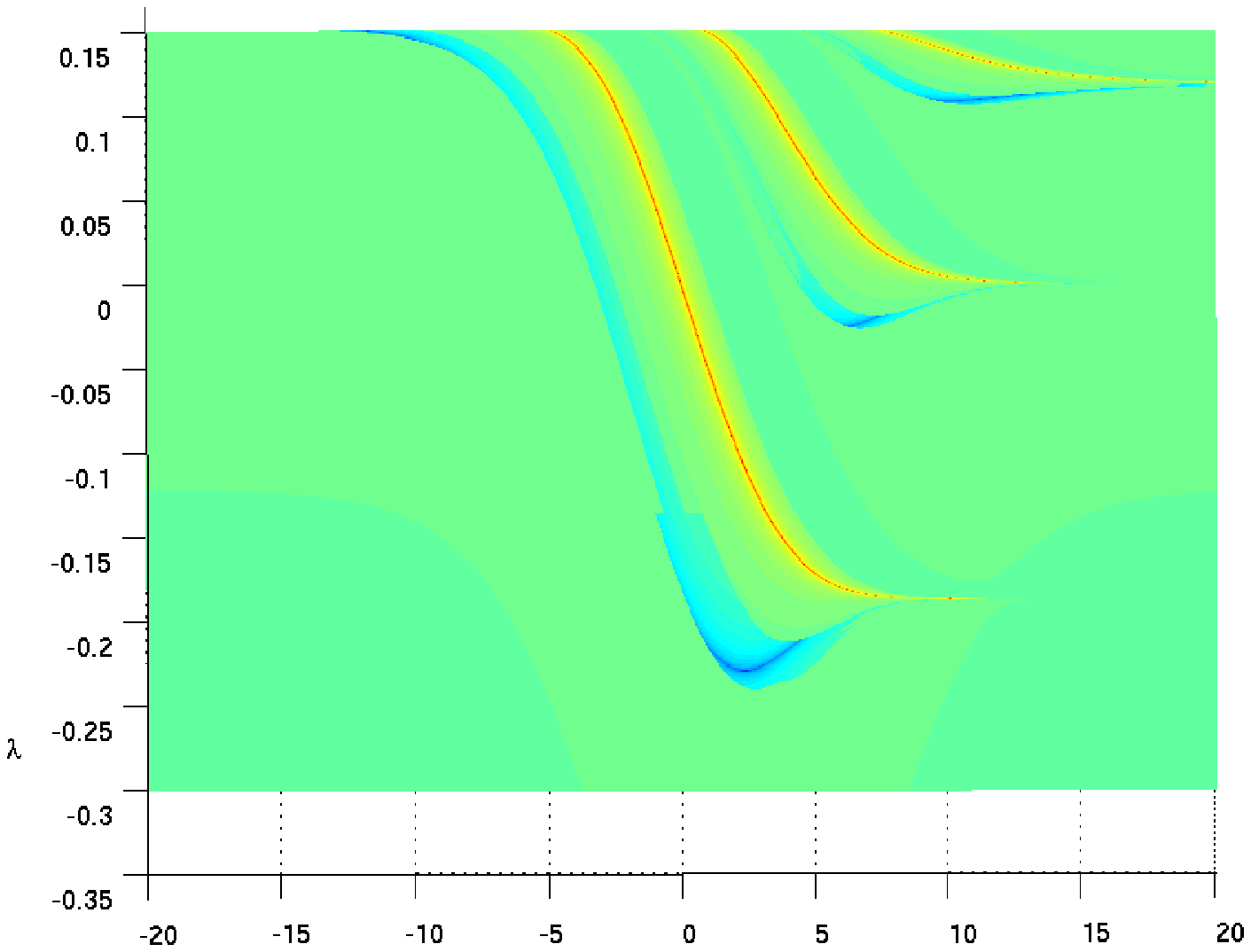}
  \includegraphics[width=6cm,height=5cm]{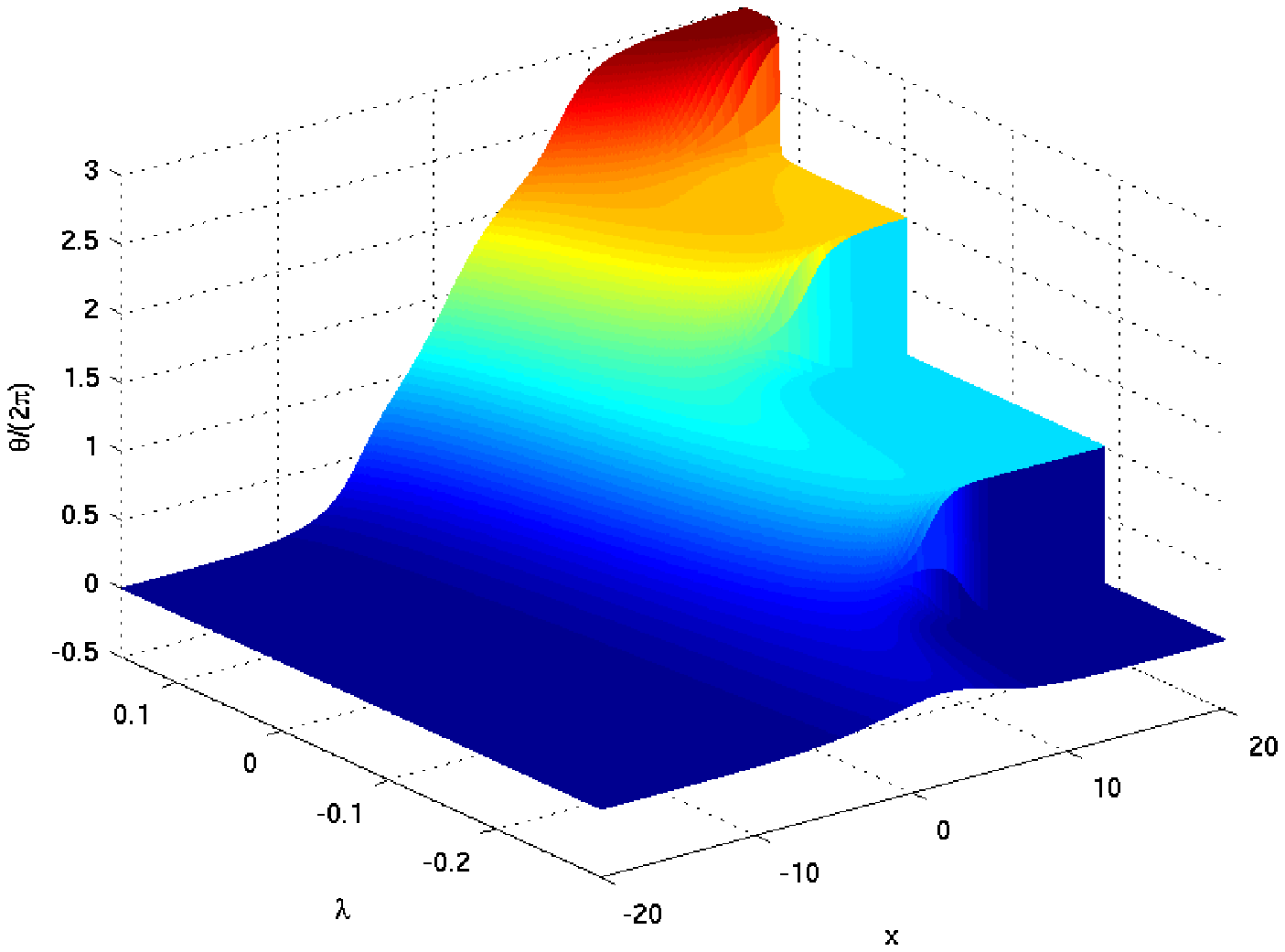}
  \end{center}
  \caption{In the left panel, we plot the singular eigenvalue 
of $s=s(x,\lambda)\in\mathrm{D}(\R^3)$. 
The other two eigenvalues of $s$ are finite. Indeed for fixed values of 
$\lambda\in[-0.3,0.15]$ as shown, we integrate the Riccati equation 
for $s$ from $x=-20$ through to
$x=20$ and record the behaviour of the eigenvalues of $s$. Red curves
indicate the eigenvalue becoming infinite while the blue curves indicate
the eigenvalue becoming zero. As expected, as $\lambda$ increases and passes 
through each spectral eigenvalue, the number of singularites in an eigenvalue of
$s$ in the interval $[-20,20]$ jumps by an integer value.
In the right panel, we plot the trace 
$\theta=\theta(x,\lambda)=-2\mathrm{i}\,\mathrm{tr}\sigma(x,\lambda)\in\mathfrak{u}(\CC^1)$ 
from Corollary~\ref{cor:trace} for $\lambda\in[-0.3,0.15]$ and 
$x\in[-20,20]$. As expected, as $\lambda$ increases and passes 
through each spectral eigenvalue, the value of $\theta(20,\lambda)$ jumps 
by an integer.}
\label{fig:singeval1top}
\end{figure}
For fixed values of $\lambda\in\R$ we also evolved the Lagrangian plane 
parameterized by $\sigma=\sigma(x,\lambda)\in\mathfrak{u}(\CC^3)$, i.e.\/
by the flow on the unitary Lie algebra given in Theorem~\ref{th:unitaryflow}.
This is the pullback of the flow on the manifold of unitary symmetric
matrices via the Lie algebra action $\ell$. 
Explicitly this is
\begin{equation*}
\pa_x\sigma=\mathrm{dexp}_\sigma^{-1}\circ
\xi\bigl(x,\exp\sigma\,u(x_0)\,\exp\sigma^{\text{\tiny{T}}}\bigr),
\end{equation*}
with data $\sigma(x_0)=O$ and where $u(x_0)$ is 
the symmetric unitary data. We chose a simple Euler method
to numerically integrate this nonlinear differential equation.
This means for a fixed $\lambda$, for the data $u_0=u_0(\lambda)$ we set
\begin{equation*}
s_0(\lambda)=p(-\infty,\lambda)q^{-1}(-\infty,\lambda)
\qquad\text{and}\qquad
u_0(\lambda)=\bigl(\id_3-\mathrm{i}s_0(\lambda)\bigr)
\bigl(\id_3+\mathrm{i}s_0(\lambda)\bigr)^{-1}.
\end{equation*}
Euler integration for a sufficiently large number of timesteps $M$ and
corresponding stepsize $h$ is given for $m=0,1,2,\ldots,M-1$ by
\begin{align*}
\sigma_{m+1}&=h\,\bigl(D(x_m,\lambda)
-\tfrac12(u_mC^\ast(x_m,\lambda)-C(x_m,\lambda)u_m^\dag)\bigr),\\
u_{m+1}&=\exp(\sigma_{m+1})\cdot u_m\cdot \exp(\sigma_{m+1}^{\text{\tiny{T}}}).
\end{align*}
The matrices $C=C(x,\lambda)$ and $D=D(x,\lambda)$ 
are explicitly defined in Lemma~\ref{lem:UnitaryRiccati}.
The real and imaginary parts of the components of $\sigma(x,0.15)$
for $x\in[-20,20]$ are shown in Figure~\ref{fig:UnitaryRiccati}. 
As we expect, the solution components are well behaved and integration
procedure very robust. In Figure~\ref{fig:singeval1top} (right panel) 
we plot the trace 
$\theta=\theta(x,\lambda)=-2\mathrm{i}\,\mathrm{tr}\sigma(x,\lambda)\in\mathfrak{u}(\CC^1)$ 
from Corollary~\ref{cor:trace} for $\lambda\in[-0.3,0.15]$ and 
$x\in[-20,20]$. As $\lambda$ increases and passes through each spectral eigenvalue, 
the value of $\theta$ jumps by an integer, contributing to the Maslov index.
\begin{figure}
  \begin{center}
  \includegraphics[width=12cm,height=4cm]{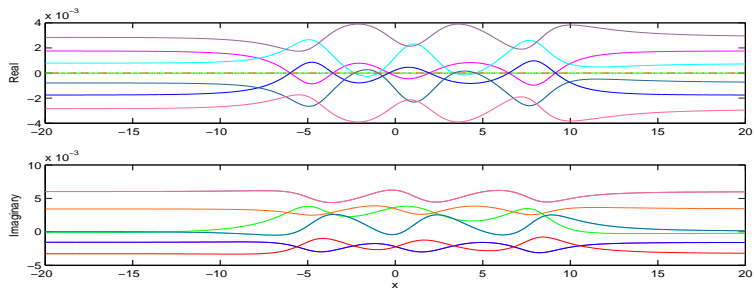}
  \end{center}
  \caption{The panels show the real (top) and imaginary (bottom) parts of
the components of the flow on the unitary Lie algebra, i.e.\/ for 
$\sigma=\sigma(x,\lambda)\in\mathfrak{u}(\CC^3)$ for $\lambda=0.15$ 
and $x\in[-20,20]$; see Theorem~\ref{th:unitaryflow}. Since $\sigma$
is anti-Hermitian there are six non-zero real parts and six distinct
imaginary parts (three are repeated).}
\label{fig:UnitaryRiccati}
\end{figure}
%

\section{Concluding remarks}\label{sec:conclu}
What about the infinite dimensional case? 
Consider a formal example to give the context and 
help clarify the picture. Recall the reaction-diffusion 
parabolic partial differential equation in the introduction. 
However now suppose it admits travelling-wave solutions on 
$\mathbb R\times\mathbb T$.       
Here we suppose the waves travel in the direction along        
$\R$ and we have transverse periodic boundary conditions indicated
by the $1$-dimensional torus $\mathbb T$ (our subsequent formal analysis
also applies to higher dimensional tori).                      
The linear stability of such waves is determined by the structure of the 
spectrum of an unbounded Fredholm operator of the form   
\begin{equation*}                                                        
H\coloneqq B(\pa_x^2+\Delta_y)+d\partial_x+V                             
\end{equation*}                                                          
for $x\in\R$ and $y\in\mathbb T$. The $n\times n$ matrix $B$           
is unchanged from our original example, the matrix $d=d(x)$ 
represents convection and the $n\times n$ potential matrix $V=V(x,y)$.                             
We assume $V(x,y)\to V_\pm(y)$, as $x\to\pm\infty$.                      
Of interest is the eigenvalue problem $Hq=\lambda q$. 
We re-write the eigenvalue problem as follows. We                        
represent the solution $q$ using a transverse Fourier--Galerkin          
basis $\{\varphi_k(y)\colon k\in\Zb\}$ where                             
$\Delta_y\varphi_k=-|k|^2\varphi_k$. Our eigenvalue                      
problem can thus be viewed as an eigenvalue                              
problem in $L^2\bigl(\R;L^2(\mathbb T;\CC^{n})\bigr)$.                   
Since $L^2(\mathbb T;\CC^{n})\cong (\ell^2)^{n}$,   
where $\ell^2$ denotes the usual Hilbert sequence space,                 
we identify each element $q=q(x)$                                        
in $L^2(\mathbb T;\CC^{n})$ with a sequence                              
$\hat q\coloneqq\{\hat q_k\colon k\in\Zb\}$ in $(\ell^2)^{\otimes n}$,   
i.e.\/ for each $x\in\R$ we replace $q(x)$ by $\hat q(x)$.  
Then in this example our separable Hilbert space                         
is $\Hb=L^2(\mathbb T;\CC^{2n})$                                         
and our phase space $L^2(\R;\Hb)$. Our eigenvalue                        
problem above now has the form $\pa^2\hat q=d\pa\hat q+c\hat q$                                   
where we replace $-B^{-1}d$ by $d$. However here we define               
\begin{equation*}                                                        
c\colon\hat q\mapsto B^{-1}(\lambda\hat q+K^2\hat q-V\star\hat q)          
\end{equation*}                                                          
where $K^2$ is the linear diagonal operator on $(\ell^2)^{\otimes n}$              
that sends each component $\hat q_k$ of $\hat q$ to $|k|^2\hat q_k$      
and $V\star\hat q$ is the usual discrete convolution product in          
$(\ell^2)^{\otimes n}$. If we now set $\hat p\coloneqq \pa\hat q$ then           
our eigenvalue problem in $L^2(\R;\Hb)$ is                               
\begin{equation*}                                                        
(\partial-A)\begin{pmatrix} \hat q \\ \hat p\end{pmatrix}=O,
\end{equation*}                                                                               
where $A=A(x;\lambda)$ has the usual block form with                                          
$a=O$, $b=\id$ and $c$ and $d$ are as defined above.                                          
Formally we can still define a compact map 
$s\colon\hat q\to\hat p$ and $s$ satifies 
the Riccati equation $\pa s=c+ds-s(a+bs)$.
Let $s_0$ denote the map 
$s_0\colon\hat q(-\infty)\to\hat p(-\infty)$. 
Then $s_0$ satisfies 
$c_0+ds_0-s_0(a+bs_0)=O$, where $c_0=c(-\infty)$.
Setting $s=s_0+\hat s$ and $c=c_0+\hat c$, then $\hat s$
satisfies the Riccati equation
\begin{equation*}                                                        
\pa\hat s=\hat c+(d-s_0b)\hat s-\hat s(a+bs_0+b\hat s).
\end{equation*}
Note the unbounded operator $K^2$ present in $c$ is not present 
in this Riccati equation.

Suppose $\Hb$ is a separable Hilbert space with a $\mathbb Z_2$-grading
$\Hb=\Qb\oplus\Pb$. 
Let $\GL(\Hb)$ denote the group of bounded invertible operators on $\Hb$.               
We define $\GL_{\mathrm{res}}(\Hb)$ to be the subgroup of $\GL(\Hb)$            
whose elements `preserve' the $\Zb_2$-grading of $\Hb$; see
Pressley and Segal~\cite{PS}. Indeed for each element 
of $\GL_{\mathrm{res}}(\Hb)$, the diagonal maps
$\Qb\to\Qb$ and $\Pb\to\Pb$ are Fredholm while the 
off-diagonal maps $s\colon\Qb\to\Pb$ and $r\colon\Pb\to\Qb$ 
are Hilbert--Schmidt. Indeed the maps $s$ parameterize a chart
of the Fredholm Grassmannian. In the symplectic case ($V$ is selfadjoint
and $d=0$ above), the Hilbert space $\Hb$ is symplectic and
subspaces $\Qb\subset\Hb$ Lagrangian. The Lagrangian Grassmannian 
$\Lambda(\Hb)$ is the space of Lagrangian subspaces. 
Its fundamental group $\pi_1\bigl(\Lambda(\Hb)\bigr)$ is trivial. 
However following Furutani~\cite{Furutani} the subset of $\Lambda(\Hb)$ 
parameterized by $s$, the Fredholm Lagrangian Grassmannian, has 
fundamental group $\mathbb{Z}$. Further, the determinant of the
Cayley map of $s$ is well defined via Fredholm (or modified) determinants.
This is the direction we intend to pursue next.


\subsection*{Acknowledgements}
We thank Graham Cox, Chris Jones, Yuri Latushkin 
Veerle Ledoux and Robbie Marangell for helpful discussions. 
MB would like to thank the Institute for Mathematics 
and its Applications at the University of Minnesota
where part of this work was conducted. 
The work of MB was partially supported by NSF DMS 1007450 
and a Sloan Fellowship.

\bibliographystyle{amsplain}

\end{document}